\documentclass{article}
\usepackage [utf8] {inputenc}
\usepackage{mathtools}
\usepackage{amsthm}
\usepackage{amssymb}
\usepackage{enumerate}

\newtheorem{theorem}{Theorem}
\newtheorem{proposition}{Proposition}
\newtheorem{lemma}{Lemma}
\newtheorem{corollary}{Corollary}

\DeclareMathOperator{\Irr}{Irr}
\DeclareMathOperator{\Cay}{Cay}

\title{Expanders and growth of normal subsets in finite simple groups of Lie type}
\author{Saveliy V. Skresanov
\thanks{The research was carried out within the framework of the Sobolev Institute of Mathematics state contract (project FWNF-2022-0002).}}
\date{}

\begin{document}
\maketitle

\begin{abstract}
	We show that some classical results on expander graphs imply growth results on normal subsets in finite simple groups.
	As one application, it is shown that given a nontrivial normal subset \( A \) of a finite simple group \( G \) of Lie type of bounded rank,
	we either have \( G \setminus \{ 1 \} \subseteq A^2 \) or \( |A^2| \geq |A|^{1+\epsilon} \), for \( \epsilon > 0 \).
	This improves a result of Gill, Pyber, Short and Szab\'o, and partially resolves a question of Pyber from the Kourovka notebook.
	We also propose a variant of Gowers' trick for two subsets, and give applications to products of large subsets in groups of Lie type,
	improving some results of Larsen, Shalev and Tiep.
\end{abstract}

\section{Introduction}

Growth of subsets in finite simple groups has been an active area of research in the recent years.
In the most general setting the problem can be stated as follows. Given two subsets \( A, B \) of a finite group \( G \), let
\[ AB = \{ ab \mid a \in A,\, b \in B \} \text{ and } A^{-1} = \{ a^{-1} \mid a \in A \} \]
denote their product and inverse, respectively. We want to know how large \( |AB| \) is in comparison to \( |A| \) and \( |B| \),
perhaps under some additional conditions like \( A = B \) and \( A^{-1} = A \).
Of course, for nonempty subsets \( |AB| \) is always at least \( \max \{ |A|, |B| \} \) and at most \( |A| \cdot |B| \), but in many cases one can prove more.

One of such cases arises in the study of Babai's conjecture on diameters of Cayley graphs of finite simple groups.
Recall that in~\cite{babaiSeress} it is conjectured that a connected Cayley graph of a nonabelian finite simple group \( G \) has polylogarithmic diameter.
That is, there exist universal constants \( C, c > 0 \) such that for any symmetric generating set \( S \subseteq G \) we have \( S^n = G \),
where \( n \leq C \cdot (\log |G|)^c \) and \( S^n = SS \cdots S \) denotes the \( n \)-fold product set.
In the breakthrough work of Helfgott~\cite{helfgott} it was shown that for \( G = \mathrm{PSL}_2(p) \), where \( p \) is a prime,
generating sets exhibit three-step growth, more precisely, we have \( |S^3| \geq |S|^{1+\epsilon} \) or \( G = S^3 \), where \( \epsilon > 0 \) is a constant.
It can be easily seen that this property immediately implies the positive solution of Babai's conjecture in the case of \( \mathrm{PSL}_2(p) \).

Similar growth theorems have been shown independently by Pyber, Szab\'o~\cite{pyberSzabo} and Breuillard, Green, Tao~\cite{BGT}
for groups of Lie type of bounded rank. It is known that such growth results do not hold for groups of Lie type of unbounded rank and
for the alternating groups (see, for instance,~\cite[Example~77]{pyberSzabo}), but if one assumes the generating set \( S \) to be \emph{normal}, i.e.\ \( S \) is a union
of conjugacy classes of \( G \), then the situation becomes much more benign.

Indeed, Babai's conjecture in the case when the generating set is normal was confirmed by Liebeck and Shalev in~\cite{liebeckShalev}.
Two-step expansion of normal subsets in finite simple groups was shown in~\cite{gillPybSS} and improved in~\cite{liebeckSchulShalev} for small subsets.
Results on expansion of two normal subsets include the fundamental results of Shalev on products of conjugacy classes~\cite{shalev},
including applications to word maps and Thompson's conjecture, see a recent survey~\cite{shalevSurvey} for many generalizations and improvements that followed.
Some expansion can be obtained when \( A \) is an arbitrary small subset of \( G \)
and \( B \) is a nontrivial normal subset, see~\cite[Proposition~5.2]{gillPybSS},~\cite[Theorem~1.3]{liebeckSchulShalev} and a recent result by Dona, Mar\'oti and Pyber~\cite{donaMP}.

In this work we will also focus on the expansion of normal subsets. A standard strategy to study the product of normal subsets
is to apply the Frobenius formula, which expresses the product of two conjugacy classes in a finite group in terms of irreducible ordinary characters,
and then to utilize various bounds on character values, such as the results on the Witten zeta function (see, for example,~\cite{witten})
and character ratio bounds (see~\cite{larsenTiep} for a recent strong bound). Here we choose a different strategy and instead appeal to some classical
results about expander graphs. Our main observation is that \emph{the character ratio
\[ R(g) = \max_{\chi \neq 1_G} \left|\frac{\chi(g)}{\chi(1)}\right|, \]
where \( g \in G \) and \( \chi \) ranges over nontrivial irreducible characters of \( G \),
controls the spectral expansion of the Cayley graph \( \Cay(G, g^G) \) over \( G \) with the connection set \( g^G \).} We formulate this observation more precisely in Lemma~\ref{specchi}
of Section~\ref{sPrelim}, but intuitively it implies that the conjugacy class \( g^G \) ``expands'' whenever \( R(g) \) is small enough.
This often holds for conjugacy classes of finite simple groups, which allows us to obtain growth results.

We note that this observation has been used before to show that normal Cayley graphs over finite simple groups of Lie type are expanders, if the
field of definition is large enough~\cite[(3.10)]{lieExpand}. Here we make one more step and show that the expansion properties of Cayley graphs have
some other group-theoretical applications.

Our first result follows from the observation on character ratios together with the well-known fact that spectral expansion implies vertex expansion
(we postpone the proofs of this and other results to Section~\ref{sProofs}).

\begin{theorem}\label{2step}
	Let \( A, B \) be nonempty subsets of a finite group \( G \) of order \( n \), and assume that \( A \) is a normal subset.
	Define
	\[ R = \min_{g \in A} R(g).  \]
	Then
	\[ |AB| \geq \frac{n}{1 + R^2(n/|B| - 1)} \geq \min \left\{ \frac{n}{2},\, \frac{|B|}{2R^2} \right\}. \]
\end{theorem}

Our main application throughout the text will be to the case when \( G \) is a finite simple group of Lie type over the field of order \( q \),
and \( A \) is a nontrivial normal subset.
By Gluck's bound~\cite{gluck}, for \( g \neq 1 \) we have \( R(g) \leq C/\sqrt{q} \) for some universal constant \( C \), so Theorem~\ref{2step} implies
that either \( |AB| \geq n/2 \) or \( |AB| \geq c \cdot q|B| \) for some constant \( c > 0 \). In fact, it follows from Gluck's work that \( R(g) \leq 19/20 < 1 \)
for any finite simple group of Lie type, so we always get some mild expansion.

In~\cite[Corollary~2.5]{babnikpy} Babai, Nikolov and Pyber proved that for arbitrary nonempty subsets \( A \) and \( B \) we have either \( |AB| \geq n/2 \)
or \( |AB| \geq m|A|/(2n) \cdot |B| \), where \( m \) is the minimum degree of a nontrivial ordinary character of \( G \).
In Section~\ref{remarks} we show that this result and some other results from~\cite{babnikpy} are also a consequence of a (different) bound
on spectral expansion of Cayley graphs. In particular, Theorem~\ref{2step} is an improvement to~\cite[Corollary~2.5]{babnikpy} when \( A \)
is a single conjugacy class of \( G \), since in that case \( R \) is precisely the spectral expansion of \( \Cay(G, A) \).

In~\cite[Theorem~3.3]{gowersTrick} Gowers showed that if \( A, B, C \subseteq G \) are such that \( |A| \cdot |B| \cdot |C| \geq |G|^3 / m \), where
\( m \) is the minimum degree of a nontrivial ordinary character of \( G \), then \( G = ABC \). This useful fact was later called the ``Gowers' trick'',
and was subsequently generalized to an arbitrary number of sets in~\cite[Corollary~2.6]{babnikpy} and recently improved for normal subsets of groups of Lie type in~\cite{marotiFr}.
Unfortunately these results do not yield a nontrivial bound for two subsets. Our next result shows that Theorem~\ref{2step} implies
such a bound.

\begin{corollary}[``Gowers' trick for two subsets'']\label{gowers2}
	Let \( A, B \) be subsets of a finite group \( G \).
	If \( |A| \cdot |B| \geq R(g)^2 \cdot |G|^2 \) for \( g \in G \), \( g \neq 1 \), then \( AB \cap g^G \neq \varnothing \).
	In particular, if \( A \) and \( B \) are normal subsets, then \( g^G \subseteq AB \).
\end{corollary}

In~\cite[Theorem~A~(iv)]{larsenShNorm} Larsen, Shalev and Tiep prove that given a finite simple group of Lie type \( G \) of bounded rank,
and two normal subsets \( A, B \subseteq G \) with \( |A|, |B| \geq \epsilon|G| \), \( \epsilon > 0 \), we have \( G \setminus \{1\} \subseteq AB \)
if \( |G| \) is sufficiently large (in terms of \( \epsilon \)). Corollary~\ref{gowers2} and Gluck's bound imply that one does not have to assume that \( G \) has bounded rank,
it is enough for the field of definition to be sufficiently large.

\begin{corollary}\label{lstImp}
	There exists a universal constant \( C > 0 \) such that the following is true.
	Let \( G \) be a finite simple group of Lie type over a field of order \( q \),
	and let \( A \) and \( B \) be normal subsets of \( G \). If \( |A| \cdot |B| \geq C/q \cdot |G|^2 \),
	then \( G \setminus \{ 1 \} \subseteq AB \).
\end{corollary}

Note that the proof of~\cite[Theorem~A~(iv)]{larsenShNorm} is rather involved and requires algebraic geometry, while our argument is elementary, aside from using Gluck's bound.

As an application of Corollary~\ref{lstImp}, we partially resolve a question of Pyber from the Kourovka notebook.
In~\cite[20.74~(a)]{kourovka} he asks if given a nonabelian finite simple group \( G \) and a normal subset \( A \)
with \( |A| > |G|/\log_2|G| \) and \( A^{-1} = A \) we have \( A^2 = G \) for \( |G| \) large enough.
In~\cite[20.74~(b)]{kourovka} the same question is posed in the particular case of \( G = \mathrm{PSL}_2(p) \).
Pyber's question was recently solved in the positive for alternating groups, see~\cite[Theorem~4]{larsenTiepAlt},
and we settle it for finite simple groups of Lie type of bounded rank.

\begin{corollary}\label{pyberConj}
	Let \( G \) be a finite simple group of Lie type of bounded rank,
	and let \( A \subseteq G \) be a normal subset satisfying \( |A| > |G|/\log_2 |G| \) and \( A^{-1} = A \).
	If \( |G| \) is large enough in terms of the rank, then \( A^2 = G \).
\end{corollary}

Note that part~(b) of Pyber's question about \( G = \mathrm{PSL}_2(p) \) follows from the above corollary.

It is often important not only to decide when a given \( g \in G \) lies in the product of subsets \( A, B \subseteq G \),
but also to obtain an asymptotic formula for the number of decompositions of the form \( g = ab \), where \( a \in A \) and \( b \in B \),
see, for instance,~\cite[Question~1.4]{shalevSurvey}.
Let \( P_{A,B}(g) \) denote the probability that \( g = ab \) where \( a \in A \), \( b \in B \) are chosen uniformly independently at random,
in other words
\[ P_{A,B}(g) = \frac{1}{|A|\cdot |B|} |\{ (a, b) \in A \times B \mid g = ab \} |. \]
Using the expander mixing lemma (Proposition~\ref{expMix}) we can improve Corollary~\ref{gowers2}
and show that \( P_{A,B}(g) \) for nontrivial \( g \) behaves close to a uniform distribution.

\begin{theorem}\label{asymp}
	Let \( A, B \) be nonempty normal subsets of a finite group \( G \).
	For \( g \in G \) we have
	\[ \left| P_{A,B}(g) - \frac{1}{|G|} \right| < \frac{R(g)}{\sqrt{|A|\cdot |B|}}. \]
\end{theorem}

Similarly to~\cite[Theorem~A~(iv)]{larsenShNorm}, in the case of groups of Lie type
we can also obtain an asymptotic formula for \( P_{A,B}(g) \), when \( |A| \geq \epsilon |G| \), \( |B| \geq \epsilon |G| \) for some \( \epsilon > 0 \).

\begin{corollary}\label{asympBig}
	Let \( G \) be a finite simple group of Lie type over the field of order~\( q \). Assume that \( A, B \subseteq G \) are normal subsets
	with \( |A| \geq \epsilon |G| \), \( |B| \geq \epsilon |G| \) for some \( \epsilon > 0 \). Then for any \( g \in G \), \( g \neq 1 \) we have
	\[ P_{A,B}(g) = \frac{1}{|G|}\left(1 + O\left(\frac{1}{\epsilon \sqrt{q}}\right)\right). \]
\end{corollary}

In the case when the normal subsets are images of group words and the group has bounded rank, we obtain a version of~\cite[Theorem~8.1~(i)]{larsenShNorm}
with a more explicit bound.

\begin{corollary}\label{words}
	Let \( G \) be a finite simple group of Lie type of bounded rank over the field of order~\( q \). Let \( w_1, w_2 \) be nonidentity group words.
	Then for any \( g \in G \), \( g \neq 1 \) we have
	\[ P_{w_1(G),w_2(G)}(g) = \frac{1}{|G|}\left(1 + O\left(\frac{1}{\sqrt{q}}\right)\right), \]
	where the implied constant in big-O depends on the rank of \( G \) only.
\end{corollary}

Our final application of results on expanders concerns the growth of squares of normal subsets in groups of Lie type.
Gill, Pyber, Short and Szab\'o~\cite{gillPybSS} proved that given a nontrivial normal subset \( A \) of a finite simple group \( G \),
we either have \( |A^2| \geq |A|^{1+\epsilon} \) or \( A^b = G \) for some universal constants \( \epsilon > 0 \), \( b \geq 1 \).
In the case of finite simple groups of Lie type of bounded rank we can take \( b = 2 \).

\begin{theorem}\label{boundrank}
	Let \( G \) be a finite simple group of Lie type, and let \( A \) be a nontrivial normal subset of \( G \).
	Then either \( |A^2| \geq |A|^{1+\epsilon} \) or \( G \setminus \{ 1 \} \subseteq A^2 \), for some constant \( \epsilon > 0 \)
	depending on the rank of \( G \) only.
\end{theorem}

Note that it is not always possible to replace the conclusion \( G \setminus \{ 1 \} \subseteq A^2 \) with \( G = A^2 \).
Recall that an element of a group is called \emph{real} if it is conjugate to its inverse.
In the case of \( G = \mathrm{PSL}_2(q) \) one can easily show that every semisimple element is real, so there exists a universal constant \( C > 0 \)
such that any normal subset \( A \subseteq G \) with \( |A| \geq C \cdot q^2 \) must contain a real element, and hence \( G = A^2 \) follows.
For \( G = \mathrm{PSL}_3(q) \) the situation is different. One can show that \( G \) contains at most \( O(q) \) conjugacy classes of real elements
(see~\cite[Table~2]{psl3table} or~\cite[Section~13]{gillSingh}) and since every conjugacy class in \( G \) has size at most \( O(q^6) \),
the group has at most \( O(q^7) \) real elements. On the other hand, \( |G| = \frac{1}{\gcd(3, q-1)} q^8 - O(q^6) \), in particular,
\( G \) contains at least \( |G|(1 - C/q) \) elements which are not real for some universal constant \( C > 0 \).
Hence there exists a normal subset \( A \subseteq G \) of size at least \( \frac{|G|}{2}(1 - C/q) \) such that \( 1 \not\in A^2 \);
note that for any \( \nu < 1/2 \) we can choose \( A \) to be bigger than \( \nu|G| \) for \( q \) large enough.
The same example also shows why it is important to exclude the identity element in Corollaries~\ref{lstImp} and~\ref{asympBig}.

In general, the abundance of real elements in a group of Lie type depends on the type of the group. In~\cite[Lemma~10]{galt} it is shown
that in finite simple groups of Lie type besides the groups of types \( A_n \), \( ^2A_n \), \( D_{2n+1} \), \( ^2D_{2n+1} \), \( E_6 \) and \( ^2E_6 \) every semisimple element is real.
By~\cite{guralnickLub}, a group of Lie type \( G \) defined over a field of order \( q \) has at least \( |G|(1 - C/q) \) semisimple elements
for some universal constant \( C > 0 \). It follows that in groups not of the six types above almost any element is real,
and hence we can strengthen Theorem~\ref{boundrank} a bit in this case, by showing that for \( A \subseteq G \) we either have \( |A^2| \geq |A|^{1+\epsilon} \)
or \( G = A^2 \).

The structure of the paper is as follows. In Section~\ref{sPrelim} we remind some basic properties of expanders and
prove the relation between character ratios and spectral expansion of normal Cayley graphs (Lemma~\ref{specchi}).
In Section~\ref{sProofs} we give the proofs of our main results and in Section~\ref{remarks} we note how some other results
on expansion in finite simple groups fit into the framework of expander graphs.

\section{Preliminaries}\label{sPrelim}

We follow~\cite[Chapter~4]{pseudo} for our exposition of expanders.
Let \( \Gamma = (V, E) \) be a directed graph (shortly, a \emph{digraph}) with the vertex set \( V \) and the arc set \( E \subseteq V \times V \).
Suppose that \( \Gamma \) is regular of valency \( d \), i.e.\ for every vertex \( x \in V \) there are exactly \( d \) arcs with
the beginning at \( x \). Set \( n = |V| \). Let \( M \) denote the \emph{random-walk matrix} of the digraph, that is,
\( M \) is a \( n \times n \) matrix where \( M_{x, y} \), \( x, y \in V \), is equal to \( 1/d \) if \( (x, y) \in E \) and \( 0 \) otherwise. Define
\[ \lambda(\Gamma) = \max_{x \perp u } \frac{\|xM\|}{\|x\|}, \]
where \( u = (1/n, \dots, 1/n) \in \mathbb{R}^n \) and \( x \in \mathbb{R}^n \) ranges over all vectors orthogonal to \( u \).
The norm \( \| \cdot \| \) denotes the Euclidean norm.

It can be easily seen that \( 1 \) is always an eigenvalue of \( M \) with an eigenvector~\( u \).
When \( \Gamma \) is an undirected graph, \( \lambda(\Gamma) \) is the second largest eigenvalue of \( M \) by absolute value, see~\cite[Lemma~2.55]{pseudo}.
By~\cite[Problem~2.10~(2)]{pseudo}, for an arbitrary digraph, \( \lambda(G) = \sqrt{\lambda_2} \), where \( \lambda_2 \) is the second largest eigenvalue of \( MM^t \); recall that this
matrix is positive semidefinite. We always have \( \lambda(\Gamma) \leq 1 \), and we call \( \lambda(\Gamma) \) the \emph{spectral expansion}
of \( \Gamma \). Note that in~\cite[Definition~4.5]{pseudo} spectral expansion is defined in terms of the spectral gap \( 1 - \lambda(\Gamma) \),
but our choice of terminology should not lead to any confusion (see the footnote to~\cite[Definition~4.5]{pseudo}).

We say that a digraph \( \Gamma \) is a \( (N, a) \) \emph{vertex expander}, if for all \( A \subseteq V \)
with \( |A| \leq N \), the neighborhood \( N(A) = \{ y \in V \mid (x, y) \in E,\, \text{ for some } x \in A \} \) satisfies \( |N(A)| \geq a|A| \).

\begin{proposition}[{Spectral expansion implies vertex expansion~\cite[Theorem~4.6]{pseudo}}]\label{specExp}
	If \( \Gamma \) is a regular digraph on \( n \) vertices, then for all \( 0 \leq \alpha \leq 1 \) the digraph \( \Gamma \) is a
	\( (\alpha n, 1/((1-\alpha)\lambda(\Gamma)^2 + \alpha)) \) vertex expander.
\end{proposition}

Given \( A, B \subseteq V \), let \( e(A, B) = |E \cap (A \times B)| \) denote the number of arcs from \( A \) to \( B \).

\begin{proposition}[{Expander mixing lemma~\cite[Lemma~4.15]{pseudo}}]\label{expMix}
	If \( \Gamma = (V, E) \) is a regular digraph of valency \( d \) on \( n \) vertices, then for all \( A, B \subseteq V \)
	with \( \alpha = |A|/n \), \( \beta = |B|/n \) we have
	\[ \left| \frac{e(A, B)}{dn} - \alpha \beta \right| \leq \lambda(\Gamma)\sqrt{\alpha(1-\alpha)\beta(1-\beta)}. \]
\end{proposition}

For a finite group \( G \) and a subset \( S \subseteq G \), let \( \Cay(G, S) \) denote the \emph{directed Cayley graph} on \( G \) with the connection set \( S \).
The set of vertices is \( G \), and we draw an arc from \( g \in G \) to \( h \in G \) if and only if \( g^{-1}h \in S \).

Let \( \Irr(G) \) denote the set of irreducible ordinary characters of a group \( G \), and let \( 1_G \) denote the trivial character.
Next result shows that spectral expansion of normal Cayley graphs is controlled by character ratios, cf.~\cite[(4.6)]{lieExpand}.

\begin{lemma}\label{specchi}
	Let \( G \) be a finite group with a normal subset \( S \subseteq G \). Then
	\[ \lambda(\Cay(G, S)) \leq \max_{g \in S} R(g). \]
	Moreover, if \( S \) is a conjugacy class, then we have an equality in the formula above.
\end{lemma}
\begin{proof}
	Let \( M \) denote the random-walk matrix of \( \Cay(G, S) \).
	Define two linear operators \( \mathcal{M}, \mathcal{M}^* : \mathbb{C}[G] \to \mathbb{C}[G] \)
	on the group algebra \( \mathbb{C}[G] \) as follows:
	\[ \mathcal{M}(h) = \frac{1}{|S|} \sum_{s \in S} s^{-1}h\, \text{ and } \mathcal{M}^*(h) = \frac{1}{|S|} \sum_{s \in S} sh, \]
	for \( h \in \mathbb{C}[G] \). It can be easily seen that the matrices of \( \mathcal{M} \) and \( \mathcal{M}^* \)
	in the standard basis are \( M \) and \( M^t \) respectively. Since \( S \) is a normal subset of \( G \), the sums \( \sum_{s \in S} s^{-1} \) and \( \sum_{s \in S} s \)
	lie in the center of the group algebra and hence \( \mathcal{M}(\mathcal{M}^*(h)) = \mathcal{M}^*(\mathcal{M}(h)) \) for all \( h \in \mathbb{C}[G] \).
	Therefore \( MM^t = M^tM \), and hence the eigenvalues of \( MM^t \) are \( |\lambda|^2 \) where \( \lambda \) runs through the eigenvalues of~\( M \).

	Set \( R = \max_{g \in S} R(g) \).
	It is well known (see, for instance,~\cite[Theorem~1]{zie}) that eigenvalues of \( M \) have the form
	\[ \lambda_\chi = \frac{1}{\chi(1)|S|} \sum_{g \in S} \chi(g) \]
	where \( \chi \) ranges over \( \Irr(G) \). The largest eigenvalue is \( \lambda_{1_G} = 1 \), so for \( \chi \neq 1_G \) we have
	\[ |\lambda_\chi| \leq \frac{1}{\chi(1)|S|} \sum_{g \in S} |\chi(g)| \leq
	\max_{g \in S} \left| \frac{\chi(g)}{\chi(1)} \right| \sum_{g \in S} \frac{1}{|S|} \leq \max_{g \in S} R(g) = R. \]
	Now, it follows that the second largest eigenvalue \( \lambda_2 \) of \( MM^t \)
	is at most \( R^2 \). Hence \( \lambda(\Cay(G, S)) = \sqrt{\lambda_2} \leq R \) as claimed.

	If \( S \) is a conjugacy class \( g^G \), then \( \lambda_\chi = \chi(g)/\chi(1) \) and hence \( \lambda(\Cay(G, S)) = R \) as wanted.
\end{proof}

The following is the main character ratio bound that we will be using.

\begin{proposition}[Gluck's bound~{\cite{gluck}}]\label{gluckB}
	Let \( G \) be a finite simple group of Lie type over the field of order~\( q \).
	For any \( \chi \in \Irr(G) \), \( \chi \neq 1_G \), and any \( x \in G \), \( x \neq 1 \)
	we have \( |\chi(x)/\chi(1)| \leq C/\sqrt{q} \) for a universal constant \( C > 0 \).
\end{proposition}

\section{Proofs of main results}\label{sProofs}

\noindent\emph{Proof of Theorem~\ref{2step}.}
Let \( g \in A \) be such that \( R = R(g) \).
Set \( \Gamma = \Cay(G, g^G) \). We have \( |AB| \geq |Bg^G| \geq |N(B)| \), where \( N(B) \) is the neighborhood
of \( B \) in \( \Gamma \). By Proposition~\ref{specExp}, the digraph \( \Gamma \) is \( (\alpha n, 1/((1-\alpha)\lambda(\Gamma)^2 + \alpha)) \)-expander
for any \( 0 \leq \alpha \leq 1 \).
Take \( \alpha = |B|/n \) and recall that \( \lambda(\Gamma) \leq R \) by Lemma~\ref{specchi}. Now we have
\[ |AB| \geq |N(B)| \geq \frac{|B|}{(1 - |B|/n)R^2 + |B|/n} = \frac{n}{1 + R^2(n/|B| - 1)}. \]
It is easy to see that the right hand side is at least \( \min\{ n/2,\, |B|/(2R^2) \} \), so the claim follows.
\qed
\medskip

\noindent\emph{Proof of Corollary~\ref{gowers2}.}
We apply Theorem~\ref{2step} to sets \( g^G \) and \( B^{-1} \). Let \( n = |G| \). Recalling that \( |B| = |B^{-1}| \) we obtain
\[ |g^G \cdot B^{-1}| \geq \frac{n}{1 + R(g)^2(n/|B|-1)} > n(1 - R(g)^2(n/|B| - 1)). \]
By assumption, \( |A| \cdot |B| \geq n^2R(g)^2 \geq nR(g)^2(n - |B|) \), hence
\[ |g^G \cdot B^{-1}| > n(1 - R(g)^2(n/|B| - 1)) \geq n - |A|. \]
Therefore \( |g^G \cdot B^{-1}| + |A| > n \), and so \( g^G \cdot B^{-1} \cap A \neq \varnothing \). It follows that \( AB \cap g^G \neq \varnothing \).
If \( A \) and \( B \) are normal, then so is \( AB \) and the final claim follows.
\qed
\medskip

\noindent\emph{Proof of Corollary~\ref{lstImp}.}
We apply Corollary~\ref{gowers2} and note that for any \( g \in G \setminus \{ 1 \} \) we have \( R(g) \leq C/\sqrt{q} \) for some universal constant \( C \)
by Proposition~\ref{gluckB}.
\qed
\medskip

\noindent\emph{Proof of Corollary~\ref{pyberConj}.}
Recall that the group \( G \) of Lie type of bounded rank defined over the field of order \( q \) has order at most \( q^c \) for some constant \( c \) depending on the rank only.
Corollary~\ref{lstImp} implies that if \( A \subseteq G \) with \( A^{-1} = A \) satisfies \( |A| \geq \sqrt{C/q} \cdot |G| \), then \( A^2 = G \).
Since \( \log_2 |G| \leq c \log_2 q \leq \sqrt{q / C} \) for \( q \) large enough, we have \( |A| > |G| / \log_2 |G| \geq \sqrt{C/q} \cdot |G| \),
so the conditions of Corollary~\ref{lstImp} are satisfied and \( A^2 = G \) as claimed.
\qed
\medskip

\noindent\emph{Proof of Theorem~\ref{asymp}.}
Consider \( \Gamma = \Cay(G, g^G) \) for \( g \in G \), and recall that \( \lambda(\Gamma) \leq R(g) \) by Lemma~\ref{specchi}.
By Proposition~\ref{expMix} applied to \( \Gamma \) and subsets \( A^{-1} \) and \( B \) we get
\[ \left| \frac{e(A^{-1}, B)}{|g^G| \cdot |G|} - \frac{|A^{-1}| \cdot |B|}{|G|^2} \right| < \frac{R(g)}{|G|} \sqrt{|A^{-1}| \cdot |B|}. \]
Notice that there is an arc between an element \( a \in A^{-1} \) and \( b \in B \) if and only if \( a^{-1}b \in g^G \).
In particular, \( e(A^{-1}, B) = |\{ (a, b) \in A \times B \mid ab \in g^G \}| \). Since \( A \) and \( B \) are normal subsets,
and since all elements of \( g^G \) are conjugate, the number of decompositions of the form \( x = ab \), \( a \in A \), \( b \in B \),
does not depend on the choice of \( x \in g^G \). Therefore
\[ P_{A,B}(g) = \frac{e(A^{-1},B)}{|g^G| \cdot |A| \cdot |B|}. \]
We derive
\[ \left| \frac{P_{A,B}(g) \cdot |A^{-1}| \cdot |B|}{|G|} - \frac{|A| \cdot |B|}{|G|^2} \right| < \frac{R(g)}{|G|} \sqrt{|A^{-1}| \cdot |B|}, \]
and the theorem follows after dividing the inequality by \( |A| \cdot |B|/|G| \) and replacing \( |A^{-1}| \) by \( |A| \).
\qed
\medskip

\noindent\emph{Proof of Corollary~\ref{asympBig}.}
Follows directly from Theorem~\ref{asymp} and Proposition~\ref{gluckB}. \qed
\medskip

\noindent\emph{Proof of Corollary~\ref{words}.}
By~\cite{larsenWords}, there exists a constant \( \epsilon > 0 \) depending on the rank of \( G \) only, such that
\( |w_i(G)| \geq \epsilon |G| \), \( i = 1, 2 \). Now the claim follows from Corollary~\ref{asympBig}. \qed
\medskip

The proof of Theorem~\ref{boundrank} will follow from a more precise result.

\begin{lemma}\label{dich}
	Let \( G \) be a nontrivial finite group, and define
	\[ R = \max_{g \in G\setminus \{ 1 \}} R(g). \]
	Let \( A \) be a nontrivial normal subset of \( G \). If \( |A| \geq R|G| \), then \( G \setminus \{ 1 \} \subseteq A^2 \).
	If \( |A| < R|G| \), then \( |A^2| \geq |A|/(2R) \).
\end{lemma}
\begin{proof}
	If \( |A| \geq R|G| \), then the claim follows from Corollary~\ref{gowers2}.
	Assume that \( |A| < R|G| \). Define \( R' = \min_{g \in A} R(g) \), and observe that \( R' \leq R \).
	By Theorem~\ref{2step}, we have
	\[ |A^2| \geq \frac{|G|}{1 + R'^2(|G|/|A| - 1)} \geq \frac{|A|}{|A|/|G| + R^2(1 - |A|/|G|)} \geq \frac{|A|}{R + R^2} \geq \frac{|A|}{2R}, \]
	and the claim follows.
\end{proof}

\noindent\emph{Proof of Theorem~\ref{boundrank}.}
Let \( G \) be a finite simple group of Lie type of rank \( r \) defined over the field of order \( q \).
We apply Lemma~\ref{dich}. If \( |A| \geq R|G| \), then \( G \setminus \{ 1 \} \subseteq A^2 \), so we may assume that \( |A| < R|G| \).
Recall that \( R \leq C/\sqrt{q} \) for some universal constant \( C \) by Proposition~\ref{gluckB},
so \( |A^2| \geq |A|/(2R) \geq |A|\sqrt{q}/(2C) \). Now, \( |G| \leq q^{cr^2} \) for some universal constant \( c \),
hence \( \sqrt{q}/(2C) \geq |G|^\epsilon \geq |A|^\epsilon \), where \( \epsilon = c'/r^2 \) for some universal constant \( c' > 0 \).
Therefore \( |A^2| \geq |A|^{1+\epsilon} \) and the claim follows.
\qed

\section{Final remarks}\label{remarks}

We have often compared our results with those obtained by Babai, Nikolov and Pyber in~\cite{babnikpy},
and it turns out that the proof method employed there is also related to expander graphs.
To state the main result of~\cite{babnikpy}, recall that given a finite group \( G \), a function
\( X : G \to \mathbb{R}_{\geq 0} \) is called a \emph{probability distribution} on~\( G \), if
\( \sum_{g \in G} X(g) = 1 \). Given two distributions \( X, Y : G \to \mathbb{R}_{\geq 0} \), their \emph{convolution} \( X * Y : G \to \mathbb{R}_{\geq 0} \)
is also a distribution defined as \( (X*Y)(h) = \sum_{g \in G} X(g)Y(g^{-1}h) \).
For a function \( f : G \to \mathbb{R} \), let \( \| f \| = \sqrt{\sum_{g \in G} |f(g)|^2 }\) denote the \( \ell_2 \)-norm.
Now, if \( U : G \to \mathbb{R}_{\geq 0} \) is the uniform distribution on \( G \), then~\cite[Theorem~2.1]{babnikpy} gives the following inequality:
\[ \| X*Y - U \| \leq \sqrt{n/m} \cdot \|X - U \| \cdot \|Y - U \|, \tag{\(*\)} \]
where \( n \) is the order of \( G \) and \( m \) is the minimum degree of a nontrivial ordinary representation of \( G \).
All the expansion results in~\cite{babnikpy} follow from this convolution inequality, for instance, to obtain two-step expansion~\cite[Corollary~2.4]{babnikpy}
(which we improve for normal subsets in Theorem~\ref{2step}), inequality~\((*)\) is applied to distributions concentrated on subsets \( A, B \subseteq G \).

We claim that this inequality is equivalent to a spectral expansion bound for a certain weighted Cayley digraph.
Keeping the same notation for the group and the distributions, let \( \Gamma \) be a complete Cayley digraph on \( G \),
where an arc between \( x \in G \) and \( y \in G \) has weight \( Y(x^{-1}y) \). The random-walk matrix \( M \) of this digraph
is a \( n \times n \) matrix with entries \( M_{x,y} = Y(x^{-1}y) \); here we identify indices used with the group elements.
By~\cite[Section~4.1.2]{pseudo}, the spectral expansion of this digraph is equal to
\[ \lambda(\Gamma) = \max_{\pi} \frac{\| \pi M - u \|}{\| \pi - u \|}, \]
where the vector \( u = (1/n, \dots, 1/n) \in \mathbb{R}^n \) corresponds to the uniform distribution,
and \( \pi \) ranges over all vectors from \( \mathbb{R}_{\geq 0}^n \) with \( \sum_g \pi_g = 1 \), in other words,
\( \pi \) ranges over arbitrary probability distributions on \( G \).
Now, if one takes \( \pi_g = X(g) \), one can easily check that \( (\pi M)_g = (X*Y)(g) \). Therefore \( \| \pi M - u \| = \| X*Y - U \| \)
and \( \| \pi - u \| = \| X - U \| \). Plugging this into the formula for \( \lambda(\Gamma) \) above and recalling that \( X \) in \( (*) \)
is an arbitrary distribution, we obtain the following bound for the spectral expansion of \( \Gamma \):
\[ \lambda(\Gamma) \leq \sqrt{n/m} \cdot \|Y - U \|. \tag{\(*'\)} \]
Since \( X \) in~\((*)\) was arbitrary, we can see that \((*)\) and \((*')\) are, in fact, equivalent.

We note that the proof of~\cite[Theorem~2.1]{babnikpy} makes use of a bound on a second largest eigenvalue of a certain matrix, see~\cite[Lemma~5.7]{babnikpy},
so the relation between~\((*)\) and spectral expansion is rather natural, albeit slightly hidden by the formulation of the main result of~\cite{babnikpy}.

To illustrate our point even further, we provide an alternative proof of two-step expansion obtained in~\cite[Corollary~2.5]{babnikpy}, using the spectral expansion.

\begin{proposition}[{Alternative proof of~\cite[Corollary~2.5]{babnikpy}}]
	Let \( A \) and \( B \) be nonempty subsets of a finite group \( G \) of order \( n \),
	and let \( m \) be the minimum degree of a nontrivial ordinary representation of \( G \). Then
	\[ |AB| > \frac{n}{1 + \frac{n^2}{m|A||B|}} \geq \min \left\{ \frac{n}{2},\, \frac{m|A||B|}{2n} \right\}. \]
\end{proposition}
\begin{proof}
	We follow the proof of Theorem~\ref{2step}. Set \( \Gamma = \Cay(G, B) \).
	The digraph \( \Gamma \) is a \( (\alpha n, (1-\alpha)\lambda(\Gamma)^2 + \alpha) \)-expander
	for \( \alpha = |A|/n \). Hence
	\[ |AB| \geq |N(A)| \geq \frac{|A|}{(1 - |A|/n)\lambda(\Gamma)^2 + |A|/n} = \frac{n}{(n/|A| - 1)\lambda(\Gamma)^2 + 1}. \]
	Let \( Y : G \to \mathbb{R}_{\geq 0} \) be a distribution on \( G \) defined as \( Y(g) = 1/|B| \) if \( g \in B \) and \( Y(g) = 0 \) otherwise.
	By~\((*')\), we have
	\[ \lambda(\Gamma) \leq \sqrt{n/m} \cdot \| Y - U \| = \sqrt{n/m}\cdot \sqrt{1/|B| - 1/n} \leq \sqrt{\frac{n}{m|B|}}. \]
	Hence
	\[ |AB| \geq \frac{n}{(n/|A| - 1)\lambda(\Gamma)^2 + 1} > \frac{n}{n/|A| \cdot n/(m|B|) + 1}, \]
	as claimed. The inequality with the minimum follows trivially.
\end{proof}

Gowers' trick for three subsets follows from this result, so it is also a corollary of expansion properties of Cayley graphs.

Of course, we do not claim that the convolution inequality~\((*)\) is rendered unnecessary by the more popular
machinery of spectral expansion of Cayley graphs. For instance, the fact that~\((*)\) works for arbitrary distributions \( X \) and \( Y \)
allows us to iterate the inequality and obtain bounds on convolutions of many distributions, see~\cite[Corollary~2.2 and Corollary~2.3]{babnikpy}.
We also note that although Lemma~\ref{specchi} gives a precise formula for spectral expansion with respect to a conjugacy class,
for arbitrary normal subsets the bound~\( (*') \) can turn out to be better.

We finally mention that some recent strong results on expansion of normal subsets are formulated in terms of convolutions, see~\cite[Theorem~1.4]{lifshHyper}.

\section{Acknowledgements}

The author would like to thank prof.\ A.A.~Gal't, L.~Pyber and A.V~Vasil'ev for fruitful discussions improving the paper considerably.

\bigskip

\noindent
\emph{Saveliy V. Skresanov}

\noindent
\emph{Sobolev Institute of Mathematics, 4 Acad. Koptyug avenue, Novosibirsk, Russia}

\noindent
\emph{Email address: skresan@math.nsc.ru}

\end{document}